\documentclass[a4paper]{amsart}

\pdfoutput=1

\usepackage{amsmath,amssymb,amsthm}
\usepackage[utf8]{inputenc}
\usepackage{mathtools}
\usepackage{scalerel}
\usepackage{multicol}
\usepackage{enumerate}
\usepackage{comment} 
\usepackage{microtype}
\usepackage{appendix}
\usepackage[dvipsnames,svgnames,x11names,hyperref]{xcolor}
\definecolor{darkred}{RGB}{139,0,0}
\definecolor{darkblue}{RGB}{0,0,139}
\definecolor{darkgreen}{RGB}{0,100,0}
\usepackage[linktocpage]{hyperref}
\hypersetup{
  colorlinks   = true, 
  urlcolor     = darkred, 
  linkcolor    = darkred, 
  citecolor   = darkgreen}
\usepackage{tikz}
\usepackage{tikz-cd}
\usepackage{bbm}
\usepackage{verbatim}

\makeatletter
\let\@wraptoccontribs\wraptoccontribs
\makeatother

\newtheorem{thm}{Theorem}[section]

\newtheorem{lem}[thm]{Lemma}

\theoremstyle{definition}

\theoremstyle{remark}
\newtheorem{rem}[thm]{Remark}
\newtheorem{example}[thm]{Example}
\numberwithin{equation}{section}

\newcommand{\bL}{\mathbb{L}}

\newcommand{\bN}{\mathbb{N}}

\newcommand{\bQ}{\mathbb{Q}}
\newcommand{\bR}{\mathbb{R}}

\newcommand{\bZ}{\mathbb{Z}}

\newcommand{\gC}{\bold{C}}

\newcommand{\gM}{\bold{M}}

\newcommand{\gR}{\bold{R}}
\newcommand{\gS}{\bold{S}}

\newcommand{\gZ}{\bold{Z}}

\newcommand{\fS}{\mathfrak{S}}

\newcommand\lra{\longrightarrow}

\newcommand\hocolim{\operatorname*{hocolim}}

\newcommand\Com{\mathbf{Com}}

\newcommand{\bk}{\mathbbm{k}}

\title{Configuration spaces as commutative monoids}

\author{Oscar Randal-Williams}
\email{o.randal-williams@dpmms.cam.ac.uk}
\address{Centre for Mathematical Sciences\\
Wilberforce Road\\
Cambridge CB3 0WB\\
UK}

\contrib[With an appendix with]{Quoc P.~Ho}
\email{maqho@ust.hk}
\address{Department of Mathematics
Hong Kong University of Science and Technology
Clear Water Bay, Hong Kong}

\begin{document}

\begin{abstract}
After 1-point compactification, the collection of all unordered configuration spaces of a manifold admits a commutative multiplication by superposition of configurations. We explain a simple (derived) presentation for this commutative monoid object. Using this presentation, one can quickly deduce Knudsen's formula for the rational cohomology of configuration spaces, prove rational homological stability, and understand how automorphisms of the manifold act on the cohomology of configuration spaces. Similar considerations reproduce the work of Farb--Wolfson--Wood on homological densities.
\end{abstract}

\maketitle

\section{Introduction}

Let $M$ be the interior of a connected compact manifold with boundary. The 1-point compactification of the space $C_n(M)$ of unordered configurations in $M$ may be written as
\begin{equation}\label{eq:Cplus}
C_n(M)^+ = \left[\frac{(M^+)^{\wedge n}}{\text{locus where two points coincide}}\right]_{\fS_n},
\end{equation}
the quotient formed in pointed spaces. Not-necessarily-disjoint union of unordered configurations defines a \emph{superposition product}
$$C_n(M)^+ \wedge C_{n'}(M)^+ \lra C_{n+n'}(M)^+$$
which is associative, commutative, and unital. This gives a unital commutative monoid object in the symmetric monoidal category $\mathsf{Top}_*^\bN$ of $\bN$-graded pointed spaces: 
$$\gC(M) : n \longmapsto C_n(M)^+.$$
The goal of this note is to explain and exploit this algebraic structure.

In the following, for a pointed space $X$ we write $X[n]$ for the $\bN$-graded pointed space which consists of $X$ in grading $n$ and the point in all other gradings, and write $\Com(Y)$ for the free unital commutative monoid on an object $Y \in \mathsf{Top}_*^\bN$.

\begin{thm}\label{thm:MainNoLabels}
There is a pushout square
\begin{equation*}
\begin{tikzcd}
\Com(M^+[2]) \rar{\epsilon} \dar{\Delta} & S^{0}[0] \dar \\
\Com(M^+[1]) \rar &\gC(M)
\end{tikzcd}
\end{equation*}
of unital commutative monoids in $\mathsf{Top}_*^\bN$, where $\epsilon$ is the augmentation and $\Delta$ is induced by the diagonal inclusion $M^+ \to [M^+ \wedge M^+]
_{\fS_2} = \Com(M^+[1])(2)$. Furthermore, this square is a homotopy pushout, i.e.\ there is an induced equivalence
$$\Com(M^+[1]) \otimes^\bL_{\Com(M^+[2])} S^{0}[0] \overset{\sim}\lra  \gC(M).$$
\end{thm}
That the square is a strict pushout of unital commutative monoids is elementary: it means identifying $\gC(M)$ as the quotient of the based symmetric power monoid of $M^+$ by the ideal given by those tuples which contain a repeated element, which is a reformulation of \eqref{eq:Cplus}. The content of the theorem is that the square is also a homotopy pushout, rendering it amenable to homological calculation.

More generally, let $\pi : L \to M$ be a vector bundle, and let
$$C_n(M ; L)^+ = \left[\frac{(L^+)^{\wedge n}}{\text{locus where two points have the same projection in $M$}}\right]_{\fS_n}.$$
These assemble in the same way to a unital commutative monoid object $\gC(M;L)$. (Of course more general spaces of labels can be implemented too, but the above suffices for us.) 

\begin{thm}\label{thm:MainC}
There is a pushout square
\begin{equation*}
\begin{tikzcd}
\Com([(L \oplus L)^+]_{\fS_2}[2]) \rar{\epsilon} \dar{\Delta} & S^{0}[0] \dar \\
\Com(L^+[1]) \rar &\gC(M;L)
\end{tikzcd}
\end{equation*}
of unital commutative monoids in $\mathsf{Top}_*^\bN$, where $\epsilon$ is the augmentation and $\Delta$ is induced by the diagonal inclusion $[(L \oplus L)^+]_{\fS_2} \to [L^+ \wedge L^+]
_{\fS_2} = \Com(L^+[1])(2)$. Furthermore, this square is a homotopy pushout, i.e.\ there is an induced equivalence
$$\Com(L^+[1]) \otimes^\bL_{\Com([(L \oplus L)^+]_{\fS_2}[2])} S^{0}[0] \overset{\sim}\lra  \gC(M;L).$$
\end{thm}
This strictly generalises Theorem \ref{thm:MainNoLabels}, which is the case where $L$ is the 0-dimensional vector bundle, so we shall mostly focus on Theorem \ref{thm:MainC} for the rest of the paper.

Recall that the derived relative tensor product may be computed by the two-sided bar construction, formed in $\mathsf{Top}_*^\bN$, so the conclusion of Theorem \ref{thm:MainC} can equivalently be stated as an equivalence
\begin{equation}\label{eq:BarConst}
B(\Com(L^+[1]), \Com([(L \oplus L)^+]_{\fS_2}[2]), S^{0}[0]) \overset{\sim} \lra\gC(M;L).
\end{equation}
This formula has many applications to the homology of configuration spaces. As one application we will show how to recover Knudsen's \cite{Knudsen} formula for $H^*(C_n(M);\bQ)$ in terms of the compactly-supported $\bQ$-cohomology of $M$ and its cup-product map, which in particular quickly implies homological stability. As another application we will show that the action on $H^*(C_n(M);\bQ)$ of the group of proper homotopy self-equivalences of $M$ factors over a surprisingly small group. Finally, in an appendix written with Quoc P.~Ho, we show how similar considerations reproduces the work of Farb--Wolfson--Wood \cite{FWW} on homological densities.

\vspace{2ex}

\noindent \textbf{Context.} This note is my attempt to give a topological implementation of some of the sheaf-theoretic ideas of Banerjee \cite{Banerjee} in the case of configuration spaces. The applications to the homology of configuration spaces given in Section \ref{sec:Applications} arise by taking singular chains of the equivalence \eqref{eq:BarConst} to obtain a derived tensor product description of the chains on $\gC(M;L)$: this description will also follow from \cite{Banerjee2} as explained in \cite[Remark 1.1]{Banerjee}. As such, the purpose of this paper is 
\begin{enumerate}[(i)]
\item to give a space-level implementation/interpretation of Banerjee's ideas in a specific case, in order to popularise them among topologists, and 
\item to explain how several classical, recent, and new results about the rational homology of configuration spaces can be obtained very efficiently from \eqref{eq:BarConst} (or its chain-level analogue).
\end{enumerate}
Everything I will describe has much to do with the work of Ho \cite{Ho, Ho2}, Petersen \cite{Petersen},  Knudsen \cite{Knudsen}, Getzler \cite{getzler, getzler2}, Kallel \cite{Kallel}, B{\"o}digheimer--Cohen--Milgram \cite{BCM}, Segal \cite{Segal}, and Arnol'd \cite{Arnold}.

\vspace{2ex}

\noindent \textbf{Acknowledgements.}  I am grateful to Andrea Bianchi, Sadok Kallel, and the anonymous referee for their useful feedback on an earlier version of the paper. ORW was supported by the ERC under the European Union’s Horizon 2020 research and innovation programme (grant agreement No.~756444).

\section{Applications}\label{sec:Applications}

\subsection{Homology of configuration spaces}
Let $M$ be $d$-dimensional. The space $C_n(M ; L)^+$ is the 1-point compactification of the $n \cdot (d + \dim(L))$-dimensional manifold 
$$C_n(M ; L) := \left[L^n \setminus\{(l_1, \ldots, l_n) \, | \, \pi(l_i) = \pi(l_j) \text{ and } i \neq j\}\right]_{\fS_n}.$$
This is a vector bundle over $C_n(M)$, but is a manifold itself and is orientable if and only if the manifold $L$ is orientable and even-dimensional. To arrange this, we can take the vector bundle $W$ given by the orientation line of $M$ plus $(d-1)$ trivial line bundles. Thus by Poincar{\'e} duality we have
$$H^*(C_n(M) ; \bk) \cong H^*(C_n(M ; W) ; \bk) \cong \widetilde{H}_{2dn - *}(C_n(M ; W)^+ ; \bk).$$

In view of this, the bar construction description \eqref{eq:BarConst} can be used, in combination with the homology of free commutative monoids (see \cite{Milgram}), to investigate $H^*(C_n(M) ; \bk)$. We do not pursue this in general here, but rather focus on the case $\bk=\bQ$, where a complete answer is possible, and reproduces a formula of Knudsen.

\subsection{Revisiting Knudsen's formula}\label{sec:Knudsen}

For an $\bN$-graded pointed space we write $H_{n,i}(X) := \widetilde{H}_i(X(n))$, and similarly for chains. Write $S^*(V)$ for the free graded-commutative algebra on a homologically graded vector space $V$, i.e.\ $S^*(V) = \bigoplus_{n \geq 0}[V^{\otimes n}]_{\fS_n}$, where the Koszul sign rule is implemented. If $V$ is equipped with additional $\bN$-grading, then this is inherited by $S^*(V)$ (but there is no Koszul sign rule associated to the $\bN$-grading, only to the homological grading).

We consider $\gC(M;W)$. There is a map $\widetilde{C}_*(W^+;\bQ)[1] \to \widetilde{C}_*(\Com(W^+[1]);\bQ)$ and, using the Eilenberg--Zilber maps, it extends to a map of cdga's
$$S^*(\widetilde{C}_*(W^+;\bQ)[1]) \lra {C}_{*,*}(\Com(W^+[1]);\bQ),$$
which is an equivalence (since the maps $[(W^+)^{\wedge n}]_{h \fS_n} \to [(W^+)^{\wedge n}]_{\fS_n}$ are rational homology isomorphisms). Similarly, there is an equivalence of cdga's
$$S^*(\widetilde{C}_*([(W \oplus W)^+]_{\fS_2};\bQ)[2]) \lra {C}_{*,*}(\Com([(W \oplus W)^+]_{\fS_2}[2]);\bQ).$$
Furthermore, one may choose formality equivalences
\begin{align*}
\widetilde{H}_*(W^+ ; \bQ) &\lra \widetilde{C}_*(W^+;\bQ)\\
 \widetilde{H}_*([(W \oplus W)^+]_{\fS_2};\bQ) &\lra \widetilde{C}_*([(W \oplus W)^+]_{\fS_2};\bQ),
\end{align*}
i.e.\ chain maps inducing the identity on homology, and hence obtain equivalences 
\begin{align*}
S^*(\widetilde{H}_*(W^+;\bQ)[1]) &\lra S^*(\widetilde{C}_*(W^+;\bQ)[1])\\
S^*(\widetilde{H}_*([(W \oplus W)^+]_{\fS_2};\bQ)[2]) &\lra S^*(\widetilde{C}_*([(W \oplus W)^+]_{\fS_2};\bQ)[2])
\end{align*}
of cdga's. In $\bN$-grading 2, the map $\Delta$ induces a map 
$$\delta_* : \widetilde{H}_*([(W \oplus W)^+]_{\fS_2} ; \bQ) \overset{\Delta_*}\lra \widetilde{H}_*([(W^+)^{\wedge 2}]_{\fS_2} ; \bQ) \cong [\widetilde{H}_*(W^+ ; \bQ)^{\otimes 2}]_{\fS_2}.$$
With these choices the square
\begin{equation*}
\begin{tikzcd}
S^*(\widetilde{H}_*([(W \oplus W)^+]_{\fS_2} ; \bQ)[2]) \dar{S^*(\delta_*)} \rar{\sim} & {C}_{*,*}(\Com([(W \oplus W)^+]_{\fS_2}[2]);\bQ) \dar{\Delta_\#}\\
S^*(\widetilde{H}_*(W^+ ; \bQ)[1]) \rar{\sim} & {C}_{*,*}(\Com(W^+[1]);\bQ)
\end{tikzcd}
\end{equation*}
need not commute, but does commute up to homotopy in the category of cdga's because the two chain maps $\widetilde{H}_*([(W \oplus W)^+]_{\fS_2} ; \bQ) \to \widetilde{C}_*([(W^+)^{\wedge 2}]_{\fS_2} ;\bQ)$ induce the same map on homology, namely $\delta_*$, so are chain homotopic. The bar construction description then gives an identification
$$\mathrm{Tor}_*^{S^*(\widetilde{H}_*([(W \oplus W)^+]_{\fS_2} ;\bQ)[2])}(S^*(\widetilde{H}_*(W^+;\bQ)[1]), \bQ[0]) \cong H_{*,*}(\gC(M;W);\bQ).$$


Recall that for a free graded-commutative algebra $S^*(V)$ on a homologically graded vector space $V$ (perhaps equipped with a further $\bN$-grading), there is a free resolution of the trivial left $S^*(V)$-module $\bQ$ given by $\epsilon : S^*(V \oplus \Sigma V) \overset{\sim}\to \bQ$ equipped with the differential given by $\partial(\Sigma v) = v$ and extended by the Leibniz rule. It is usually called the Koszul resolution. It is indeed a resolution because it is the free graded-commutative algebra on the acyclic chain complex $\Sigma V \overset{id}\to V$, and over $\bQ$ taking homology commutes with the formation of symmetric powers. Applying this resolution to calculate the Tor groups above gives the complex
\begin{equation*}
\left(S^*\big(\widetilde{H}_*(W^+;\bQ)[1] \oplus \Sigma \widetilde{H}_{*}([(W \oplus W)^+]_{\fS_2} ;\bQ)[2]\big) , \partial \right)
\end{equation*}
with differential given by $\partial(\Sigma x) = \Delta_*(x) \in S^2(\widetilde{H}_*(W^+[1];\bQ))$ for $x \in \widetilde{H}_*([(W \oplus W)^+]/\fS_2[2];\bQ)$, and extended by the Leibniz rule. This can be simplified as follows. If $M$ is $d$-dimensional then the Thom isomorphism gives $\widetilde{H}_*(W^+;\bQ) = \Sigma^d\widetilde{H}_*(M^+ ; \bQ^{w_1})$, where $\bQ^{w_1}$ is the orientation local system of $M$. It also gives $\widetilde{H}_*((W \oplus W)^+;\bQ) = \Sigma^{2d}\widetilde{H}_*(M^+;\bQ)$. The involution swapping the two $W$ factors acts as $(-1)^d$ on the Thom class, so because the map $[(W \oplus W)^+]_{h\fS_2} \to [(W \oplus W)^+]_{\fS_2}$ is a rational equivalence we find
$$\widetilde{H}_*([(W \oplus W)^+]_{\fS_2};\bQ) = \begin{cases}
\Sigma^{2d}\widetilde{H}_*(M^+;\bQ) & d \text{ even}\\
0 & d \text{ odd}.
\end{cases}$$
This lets us write the complex as
\begin{equation}\label{eq:CECx}
\left(S^*\big( \Sigma^d\widetilde{H}_*(M^+ ; \bQ^{w_1})[1] \oplus \begin{cases}
\Sigma^{2d+1}\widetilde{H}_*(M^+;\bQ) & d \text{ even}\\
0 & d \text{ odd}
\end{cases}[2] \big) , \partial \right),
\end{equation}
where the differential is dual to the map $S^2(H^*_c(M ; \bQ^{w_1})) \to H^*_c(M ; \bQ)$ induced by cup product, so following Knudsen we can recognise this complex as the Chevelley--Eilenberg complex for the bigraded Lie algebra $H^*_c(M;\mathrm{Lie}(\Sigma^{d-1}\bQ^{w_1}[1]))$. Thus 
$$H^{2nd-*}(C_n(M);\bQ) \cong \widetilde{H}_{*}(C_n(M;W)^+;\bQ) \cong H^*_{\mathrm{Lie}}(H^*_c(M;\mathrm{Lie}(\Sigma^{d-1}\bQ^{w_1}[1])))(n).$$
After appropriate dualisations and reindexings, this agrees with Knudsen's formula.

\subsection{Homological stability}\label{sec:HomStab}

Stability for the homology of configuration spaces is by now a classical subject, with a large number of contributions by many authors: notable examples are \cite{Arnold, Segal, Church, RWConf, BenderskyMiller, CanteroPalmer,KupersMiller, Knudsen}. In particular Knudsen has explained \cite[Section 5.3]{Knudsen} how his formula implies rational (co)homological stability for the spaces $C_n(M)$. Let us briefly review this from the point of view taken here.

There is a canonical element $[M] \in \widetilde{H}_d(M^+ ; \bQ^{w_1}) \cong \widetilde{H}_{2d}(W^+;\bQ)$, and choosing a cycle representing this element provides a map 
$$\sigma : \Sigma^{2d}\bQ[1] \lra {C}_{*,*}(\Com(W^+[1]);\bQ) \lra {C}_{*,*}(\gC(M;W);\bQ).$$ Multiplication by this element defines a map
$$(\sigma \cdot -)_* : \widetilde{H}_{n-1,2d(n-1)-i}(\gC(M;W);\bQ) \lra \widetilde{H}_{n,2dn-i }(\gC(M;W);\bQ)$$
which under Poincar{\'e} duality gives a map $H^i(C_{n-1}(M);\bQ) \to H^i(C_{n}(M);\bQ)$; this can be checked to be the transfer map which sums over all ways of forgetting one of the $n$ points, see \cite[Section 5.2]{Knudsen} \cite[Section 2.6]{StavrouThesis}.

Writing ${C}_{*,*}(\gC(M;W);\bQ)/\sigma$ for the mapping cone of left multiplication by $\sigma$, the discussion above shows that its homology is calculated by a complex
\begin{equation*}
\left(S^*\big( \Sigma^d\frac{\widetilde{H}_*(M^+ ; \bQ^{w_1})}{\langle [M] \rangle}[1] \oplus \begin{cases}
\Sigma^{2d+1}\widetilde{H}_*(M^+;\bQ) & d \text{ even}\\
0 & d \text{ odd}
\end{cases}[2] \big) , \partial \right).
\end{equation*}
As $M$ is connected, if we assume that $d \geq 3$ then the bigraded vector spaces $\Sigma^d\frac{\widetilde{H}_*(M^+ ; \bQ^{w_1})}{\langle [M] \rangle}[1]$ and $\Sigma^{2d+1}\widetilde{H}_*(M^+;\bQ)[2]$ both vanish in bidegrees $(n,j)$ satisfying $j > (2d-1)n$, and hence so does the free graded-commutative algebra on them. This translates to $H^i(C_{n-1}(M);\bQ) \to H^i(C_{n}(M);\bQ)$ being surjective for $i<n$ and an isomorphism for $i < n-1$. For $d=2$ the same considerations give surjectivity for $i < \tfrac{1}{2}n$ and so on (a more careful analysis gives a slope 1 range in this case too, see \cite[Proof of Theorem 1.3]{Knudsen}). 

Analysing the complex \eqref{eq:CECx} can also establish other kinds of stability results, e.g.\ \cite{BerceanuYameen, KnudsenMillerTosteson, Yameen}.

\subsection{The action of automorphisms on unordered configurations}

Using Knudsen's formula it is possible to mislead yourself into thinking that homeomorphisms of $M$ (or indeed pointed homotopy self-equivalences of $M^+$) act on $H_*(C_n(M);\bQ)$ via their action on $H_*(M;\bQ)$: in other words, that such maps which act trivially on the homology of $M$ also act trivially on the homology of $C_n(M)$. This is not true: in the case of surfaces see Bianchi \cite[Section 7]{BianchiSplit}, Looijenga \cite{Looijenga}, and the complete analysis given by Stavrou \cite{Stavrou}.

From the point of view taken here this phenomenon can be explained as follows. For simplicity suppose that $M$ is orientable, and first suppose that it is odd-dimensional. Then $H^*(C_n(M);\bQ) \cong \widetilde{H}_{2dn-*}(C_n(M ; M \times \bR^d)^+;\bQ)$ and the analysis of Section \ref{sec:Knudsen} applied to $\gC(M ; M \times \bR^d)$ shows that $\Com(S^d \wedge M^+[1]) \to \gC(M ; M \times \bR^d)$ is a rational homology isomorphism. So we find: 

\begin{thm}
If $M$ is orientable and odd-dimensional, then a pointed homotopy self-equivalence of $M^+$ which acts trivially on $\widetilde{H}_*(M^+;\bQ)$ also acts trivially on $H^*(C_n(M);\bQ)$.\qed
\end{thm}

The even-dimensional case is more interesting. As $M$ is assumed orientable, in this case the twisting by $W$ can be dispensed with. It is technically convenient here---for reasons of symmetric monoidality---to work in the category of simplicial $\bQ$-modules rather than chain complexes. We write $-\odot-$ for the tensoring of this category over simplicial sets. For a space $X$ let us abbreviate $\bQ[X] := \bQ[\mathrm{Sing}_\bullet(X)]$, and if it is based then let $\widetilde{\bQ}[X] = \bQ[X]/\bQ[*]$. The discussion in the previous section, ignoring the formality step and translated to simplicial $\bQ$-modules, shows that given the simplicial module $\widetilde{\bQ}[M^+]$ and the map $\delta : \widetilde{\bQ}[M^+] \to \big[ \widetilde{\bQ}[M^+]^{\otimes 2} \big]_{\fS_2}$ induced by the diagonal $M^+ \to M^+ \wedge M^+$, we may form the two-sided bar construction
\begin{equation}\label{eq:TwoSidedBar}
B(S^*(\widetilde{\bQ}[M^+][1]) , S^*(\widetilde{\bQ}[M^+][2]), \bQ[0])
\end{equation}
whose bigraded homotopy groups are identified with $\widetilde{H}_*(C_*(M)^+;\bQ)$.

A homeomorphism of $M$, or a pointed homotopy self-equivalence of $M^+$, induces an equivalence $\phi : \widetilde{\bQ}[M^+] \to \widetilde{\bQ}[M^+]$ such that $\delta \circ \phi = ([\phi \otimes \phi]_{\fS_2}) \circ \delta$, meaning that the diagram of simplicial commutative rings
\begin{equation*}
\begin{tikzcd}
S^*(\widetilde{\bQ}[M^+][1]) \dar{S^*(\phi)} & S^*(\widetilde{\bQ}[M^+][2]) \lar[swap]{S^*(\delta)} \rar{\epsilon} \dar{S^*(\phi)} & \bQ[0] \dar{id}\\
S^*(\widetilde{\bQ}[M^+][1]) & S^*(\widetilde{\bQ}[M^+][2]) \lar[swap]{S^*(\delta)}  \rar{\epsilon}  & \bQ[0]
\end{tikzcd}
\end{equation*}
is commutative, which induces a self-equivalence of the two-sided bar construction \eqref{eq:TwoSidedBar}. This corresponds to the induced action on $\widetilde{H}_*(C_*(M)^+;\bQ)$.

However a weaker kind of data suffices to get an induced equivalence on two-sided bar constructions. An equivalence $\phi : \widetilde{\bQ}[M^+] \to \widetilde{\bQ}[M^+]$ together with a homotopy $h : \delta \circ \phi \Rightarrow ([\phi \otimes \phi]_{\fS_2}) \circ \delta$ gives a diagram of simplicial commutative rings as above where the right-hand square commutes and the left-hand squares commutes up to the homotopy $S^*(h) : S^*(\Delta^1 \odot \widetilde{\bQ}[M^+][2]) \to S^*(\widetilde{\bQ}[M^+][1])$. This data suffices to obtain a self-equivalence $\chi(\phi, h)$ of the two-sided bar construction \eqref{eq:TwoSidedBar}, as the zig-zag
\begin{equation*}
\tikzcdset{scale cd/.style={every label/.append style={scale=#1},
    cells={nodes={scale=#1}}}}
\begin{tikzcd}[scale cd=0.9]
 B(S^*(\widetilde{\bQ}[M^+][1]) , S^*(\Delta^1 \odot \widetilde{\bQ}[M^+][2]), \bQ[0])   & B(S^*(\widetilde{\bQ}[M^+][1])' , S^*(\widetilde{\bQ}[M^+][2]), \bQ[0]) \arrow[l, "{B(id, S^*(d_0), id)}" {yshift=-5pt}] \arrow[d, "{B(id, S^*(\phi), id)}"]\\
 B(S^*(\widetilde{\bQ}[M^+][1]) , S^*(\widetilde{\bQ}[M^+][2]), \bQ[0]) \arrow[u, "{B(S^*(\phi), S^*(d_1), id)}"] \arrow[r, dashed, "{\chi(\phi, h)}"]& B(S^*(\widetilde{\bQ}[M^+][1]) , S^*(\widetilde{\bQ}[M^+][2]), \bQ[0]) 
\end{tikzcd}
\end{equation*}
where $S^*(\widetilde{\bQ}[M^+][1])'$ denotes $S^*(\widetilde{\bQ}[M^+][1])$ considered as a $S^*(\widetilde{\bQ}[M^+][2])$-module via $S^*(\phi) \circ S^*(\delta)$. 

Let $(\phi', h')$ be another such datum, and suppose that there is a homotopy $\Phi : \phi \Rightarrow \phi'$ such that the 2-cell
\begin{equation}\label{eq:2cell}
\begin{tikzcd}
{[\widetilde{\bQ}[M^+]^{\otimes 2}]_{\fS_2}}  \arrow[ddd, swap, "{[(\phi')^{\otimes 2}]_{\fS_2}}", ""{name=D1}] & & &\widetilde{\bQ}[M^+] \arrow[ddd, "\phi'", ""{name=D4}] \arrow[lll, swap, "\delta",""{name=D6}]\\
& {[\widetilde{\bQ}[M^+]^{\otimes 2}]_{\fS_2}} \arrow[lu, equals] \arrow[d, swap, "{[\phi^{\otimes 2}]_{\fS_2}}",""{name=D2}] &\widetilde{\bQ}[M^+] \arrow[l, swap, "\delta", ""{name=D5}] \arrow[d, "\phi", ""{name=D3}] \arrow[ru, equals] \\
& {[\widetilde{\bQ}[M^+]^{\otimes 2}]_{\fS_2}} \arrow[ld, equals]  &\widetilde{\bQ}[M^+] \arrow[l, "\delta", ""{name=D7}] \arrow[rd, equals]\\
{[\widetilde{\bQ}[M^+]^{\otimes 2}]_{\fS_2}} & & &\widetilde{\bQ}[M^+] \arrow[lll, "\delta", ""{name=D8}]
\arrow[phantom,from=D1,to=D2,"{[\Phi^{\otimes 2}]_{\fS_2} \quad}"]
\arrow[phantom,from=D3,to=D4,"{\Phi}"]
\arrow[phantom,from=D2,to=D3,"{h}"]
\arrow[phantom,from=D5,to=D6,"{Id}"]
\arrow[phantom,from=D7,to=D8,"{Id}"]
\end{tikzcd}
\end{equation}
is homotopic to $h'$. Then one may check that $\chi(\phi', h')$ is homotopic to $\chi(\phi, h)$. If we let $\Gamma$ denote the set of $(\phi, h)$'s modulo the equivalence relation $(\phi, h) \sim (\phi', h')$ when there exists a homotopy $\Phi$ having the above property, then composition of maps and pasting of homotopies makes $\Gamma$ into a group, which acts on the two-sided bar construction \eqref{eq:TwoSidedBar} in the homotopy category of simplicial $\bQ$-modules (and so also acts on its homotopy groups). A pointed homotopy self-equivalence of $M^+$ acts on the two-sided bar construction through $\Gamma$, via elements of the special form $[(\phi, Id)]$.

We may analyse the group $\Gamma$ as follows. There is a homomorphism
\begin{align*}
\rho: \Gamma &\lra \mathrm{Aut}(\widetilde{\bQ}[M^+])\\
[(\phi, h)] &\longmapsto [\phi]
\end{align*}
to the group of homotopy classes of homotopy self-equivalences of $\widetilde{\bQ}[M^+]$. Using the Dold--Kan theorem the latter can be identified with the group of homotopy classes of homotopy self-equivalences of $\widetilde{C}_*(M^+;\bQ)$, and using a formality equivalence $\widetilde{H}_*(M^+;\bQ) \overset{\sim}\to \widetilde{C}_*(M^+;\bQ)$ this is identified with the group $\mathrm{Aut}(\widetilde{H}_*(M^+;\bQ))$ of automorphisms of the graded vector space $\widetilde{H}_*(M^+;\bQ)$. Such an automorphism is in the image of $\rho$ precisely when it preserves the map $\delta_* : \widetilde{H}_*(M^+;\bQ) \to [\widetilde{H}_*(M^+;\bQ)^{\otimes 2}]_{\fS_2}$. The kernel of $\rho$ consists of those $[(\phi, h)]$ such that $\phi$ is homotopic to the identity: by definition of the equivalence relation $\sim$ such an element may be written as $[(Id, h')]$ where $h'$ is obtained from $h$ and a homotopy $\Phi : \phi \Rightarrow id$ by the 2-cell diagram \eqref{eq:2cell}. Such an $h'$ is a self-homotopy of the map $\delta$, so an element of $\pi_1(\mathrm{map}(\widetilde{\bQ}[M^+], \big[\widetilde{\bQ}[M^+]^{\otimes 2}\big]_{\fS_2})) ; \delta)$. The ambiguity in $h'$ when representing  $[(\phi, h)] \in \mathrm{Ker}(\rho)$ as $[(Id, h')]$ comes from the choice of the homotopy $\Phi$, so $h'$ is well-defined modulo the ambiguity coming from the self-homotopies $\pi_1(\mathrm{map}(\widetilde{\bQ}[M^+], \widetilde{\bQ}[M^+]); id)$ of the identity map. In conclusion, this discussion establishes an exact sequence 
\begin{equation*}
\begin{tikzcd}
 \pi_1(\mathrm{map}(\widetilde{\bQ}[M^+], \widetilde{\bQ}[M^+]); id) \rar{\delta \circ -}
             \ar[draw=none]{d}[name=X, anchor=center]{}
    & \pi_1(\mathrm{map}(\widetilde{\bQ}[M^+], \big[\widetilde{\bQ}[M^+]^{\otimes 2}\big]_{\fS_2});\delta) \ar[rounded corners,
            to path={ -- ([xshift=2ex]\tikztostart.east)
                      |- (X.center) \tikztonodes
                      -| ([xshift=-2ex]\tikztotarget.west)
                      -- (\tikztotarget)}]{dl}[at end]{} \\      
   \Gamma \rar{\rho} & \mathrm{Aut}(\widetilde{H}_*(M^+;\bQ), \delta_*) \rar & 1.
\end{tikzcd}
\end{equation*}
Using the Dold--Kan theorem and a formality equivalence again we can identify the first map in this sequence with
$$\delta_* \circ - : \mathrm{Hom}(\Sigma \widetilde{H}_*(M^+;\bQ), \widetilde{H}_*(M^+;\bQ)) \lra \mathrm{Hom}(\Sigma \widetilde{H}_*(M^+;\bQ), [\widetilde{H}_*(M^+;\bQ)^{\otimes 2}]_{\fS_2}),$$
and so describe $\Gamma$ by an extension
$$1 \to \mathrm{Hom}(\Sigma \widetilde{H}_*(M^+;\bQ), S^2(\widetilde{H}_*(M^+;\bQ))/\mathrm{Im}(\delta_*)) \to \Gamma \to \mathrm{Aut}(\widetilde{H}_*(M^+;\bQ), \delta_*) \to 1.$$

This implies the following. We continue to assume that $M$ is even-dimensional and orientable. Let $G$ denote the group of homotopy classes of pointed homotopy self-equivalences of $M^+$ which act as the identity on $\widetilde{H}_*(M^+;\bQ)$.

\begin{thm}\label{thm:action}
If $M$ is orientable and even-dimensional, then $G$ acts on ${H}^*(C_n(M);\bQ)$ via 
$\mathrm{Hom}(\Sigma \widetilde{H}_*(M^+;\bQ), S^2(\widetilde{H}_*(M^+;\bQ))/\mathrm{Im}(\delta_*))$.\qed
\end{thm}

\begin{example}
When $M$ is a punctured surface one has $\widetilde{H}_*(M^+;\bQ) = \Sigma H_1(M;\bQ) \oplus \Sigma^2\bQ$ so the map $\delta_* : \widetilde{H}_*(M^+;\bQ) \to S^2(\widetilde{H}_*(M^+;\bQ))$ has the form
$$\Sigma H_1(M;\bQ) \oplus \Sigma^2\bQ \lra \Sigma^2 \Lambda^2(H_1(M;\bQ)) \oplus \Sigma^3 H_1(M;\bQ) \oplus \Sigma^4 \bQ,$$
which in grading 2 is the inclusion of the symplectic form $\omega \in \Lambda^2(H_1(M;\bQ))$ and is zero otherwise. Thus the above is $\mathrm{Hom}(H_1(M;\bQ), \Lambda^2(H_1(M;\bQ))/\langle \omega \rangle) \oplus H_1(M;\bQ)$. Using Poincar{\'e} duality and $\Lambda^2(H_1(M;\bQ)) \cong \bQ\{\omega\} \oplus \Lambda^2(H_1(M;\bQ))/\langle \omega \rangle$, this can be identified with $\mathrm{Hom}(H_1(M;\bQ), \Lambda^2(H_1(M;\bQ)))$. This is the target of the Johnson homomorphism, cf.\ \cite{Stavrou}.
\end{example}

\begin{rem}
The results of this section should also follow from \cite[Theorem 1.2]{Stavrou} and some rational homotopy theory.
\end{rem}

\section{Proof of Theorem \ref{thm:MainC}}

Recall that $X \in \mathsf{Top}_*$ is \emph{well-based} if the basepoint map $i : * \to X$ is a closed cofibration: under this condition $X \wedge -$ preserves weak equivalences between well-based spaces, and preserves closed cofibrations. Let us say that an $\bN$-graded based space $Y$ is well-based if $Y(n)$ is well-based for each $n \in \bN$. 

Let us write $\gR := \Com(L^+[1])$ and $\gS := \Com([(L \oplus L)^+]_{\fS_2}[2])$ to ease notation, so $\Delta : \gS \to \gR$ makes $\gR$ into a $\gS$-module. 

\begin{lem}\label{lem:cofibrancy}
$\gS$ and $\gR$ are well-based. The subspace of $[(L^+)^{\wedge p}]_{\fS_p}$ of those tuples which do not have distinct $M$ coordinates is well-based, and this inclusion is a closed cofibration.
\end{lem}
\begin{proof}
Recall that $M$ is the interior of a compact manifold with boundary $\overline{M}$. This admits a collar, showing that $i : M \to \overline{M}$ admits a homotopy inverse, and so the vector bundle $L \to M$ extends to a vector bundle over $\overline{M}$, which we also call $L$. Furthermore, choosing an inner product on this bundle we can form the closed disc bundle $D(L) \to \overline{M}$, and consider $L$ as lying inside it as the open disc bundle. Now $D(L)$ is a manifold with boundary $\partial D(L) = S(L) \cup D(L)\vert_{\partial \overline{M}}$, and $L^+ = D(L)/\partial D(L)$.

Observe that $(\overline{M}, \partial \overline{M})$ is an compact manifold pair so (is an ENR pair and hence) can be expressed as a retract of a pair $(|X_\bullet|, |\partial X_\bullet|)$ of the geometric realisations of a simplicial set and a subset. We may pull $L$ back to $|X_\bullet|$ using the retraction; let us call this $L_X$. Now $D(L_X)/S(L_X) \cup D(L_X)\vert_{|\partial X_\bullet|}$ can be given an evident cell-structure (by induction over the relative cells of $|\partial X_\bullet| \to |X_\bullet|$), and $L^+ = D(L)/\partial D(L)$ is a retract of it, so is well-based. More generally, for the exterior direct sum $L_X^{\boxplus p} \to |X_\bullet^p|$ and writing $\partial |X_\bullet^p|$ for the subcomplex where some factor lies in $\partial X_\bullet$, there is a cell structure on $D(L_X^{\boxplus p})/S(L_X^{\boxplus p}) \cup D(L_X^{\boxplus p})\vert_{\partial |X_\bullet^p|}$ for which the group $\fS_p$ acts cellularly, and so $[D(L_X^{\boxplus p})/S(L_X^{\boxplus p}) \cup D(L_X^{\boxplus p})\vert_{\partial |X_\bullet^p|}]_{\fS_p}$ is a cell complex of which $[(L^+)^{\wedge p}]_{\fS_p}$ is a retract, and so is well-based. This shows that $\gR$ is well-based, and similar reasoning shows $\gS$ is.

For the second statement, 
$$inc: F := \text{fat diagonal of } |X_\bullet|^p = |\text{fat diagonal of }X_\bullet^{p}| \lra |X_\bullet^{p}| = |X_\bullet|^p$$
is the inclusion of a $\fS_p$-CW-subcomplex, and so has a $\fS_p$-equivariant open neighbourhood $U$ which equivariantly deformation retracts to it. This may be chosen to preserve the subcomplexes where some factor lies in $|\partial X_\bullet|$. Thus it lifts to a $\fS_p$-equivariant deformation retraction of an open neighbourhood of $L_X^{\boxplus p}\vert_F \to L_X^{\boxplus p}$, and descends to the quotient by the subcomplexes where some factor lies in $|\partial X_\bullet|$. As it is equivariant, it descends further to the $\fS_p$-quotient. That is, it proves the claim for $(\overline{M}, \partial \overline{M}, L)$ replaced by $(|X_\bullet|, |\partial X_\bullet|, L_X)$; as the former data is a retract of the latter, the claim follows.
\end{proof}

\begin{lem}\label{lem:cof}
$\gR$ is a flat $\gS$-module, in the sense that $\gR \otimes_\gS -$ preserves weak equivalences between left $\gS$-modules whose underlying objects are well-based.
\end{lem}
\begin{proof}
Recall that $\gR(n) = [(L^+)^{\wedge n}]_{\fS_n}$. Define a filtration of $\gR$ by $F_0 \gR = \gS$ and
$$F_p \gR(n) := F_{p-1} \gR(n) \cup \mathrm{Im}\big( (L^+)^{\wedge p} \wedge ((L \oplus L)^+)^{\wedge (n-p)/2}  \to \gR(n)\big),$$
where the latter term is only taken when it makes sense: for $n-p$ even. This is a filtration by right $\gS$-modules. One checks that the diagram
\begin{equation*}
\begin{tikzcd}
 F_{p-2} \gR(p)[p] \otimes \gS  \rar \dar & F_{p-1}\gR \dar\\
 \gR(p)[p] \otimes \gS \rar & F_{p}\gR
\end{tikzcd}
\end{equation*}
is a pushout (in $\mathsf{Top}_*^\bN$ and so in right $\gS$-modules), where the horizontal maps are induced by the $\gS$-module structure and the adjoints of the map $inc: F_{p-2}\gR(p) \to F_{p-1}\gR(p)$, and the map $id : \gR(p) \to F_p \gR(p)$.

We prove by induction on $p$ that $F_p \gR$ is a flat $\gS$-module in the indicated sense. As $F_0 \gR = \gS$ these properties hold for $p=0$. For $\gM$ a left $\gS$-module whose underlying object is well-based, applying $- \otimes_\gS \gM$ to the square above gives a pushout square 
\begin{equation}\label{eq:Filt}
\begin{tikzcd}
F_{p-2} \gR(p)[p] \otimes \gM  \rar \dar & F_{p-1}\gR \otimes_\gS \gM \dar\\
\gR(p)[p] \otimes \gM \rar & F_{p}\gR\otimes_\gS \gM.
\end{tikzcd}
\end{equation}
The map $F_{p-2} \gR(p) \to \gR(p)$ is the inclusion of the subspace of those $p$-tuples of points in $M$ labelled by $L$ which do not have distinct $M$ coordinates, so is a closed cofibration from a well-based space by the second part of Lemma \ref{lem:cofibrancy}. As $\gM$ is assumed well-based, the left-hand vertical map in \eqref{eq:Filt} is a closed cofibration in each grading, and so this square is also a homotopy pushout. A weak equivalence $f : \gM \overset{\sim}\to \gM'$ then induces a map of homotopy pushout squares which is a weak equivalence on all but the bottom right corner, by inductive assumption, so also induces a weak equivalence on this corner.

Thus each $F_p \gR$ is flat in the indicated sense, so $\gR$ is too because $F_p \gR \to \gR$ is an isomorphism when evaluated on $n < p$, so $F_p \gR \otimes_\gS \gM \to \gR\otimes_\gS \gM$ is too.
\end{proof}

\begin{rem}
In the case that $L$ is the $0$-dimensional vector bundle, the filtration stage $F_p\gR(n)$ consists of those elements in the $n$th based symmetric power $[(M^+)^{\wedge n}]_{\fS_n}$ containing at most $p$ unrepeated elements. Up to reindexing, this is the same as the filtration used by Arnol'd \cite{Arnold} and by Segal \cite{Segal}.
\end{rem}

\begin{lem}\label{lem:tensor}
The induced map $\gR \otimes_\gS S^{0}[0] \to \gC(M;L)$ is an isomorphism.
\end{lem}
\begin{proof}
By definition of the relative tensor product there is a coequaliser diagram
\begin{equation*}
\begin{tikzcd}
\gR \otimes \gS  \arrow[r, shift left=.75ex, "\alpha"] \arrow[r,shift right=.75ex, swap, "\beta"]& \gR  \rar & \gR \otimes_\gS S^{0}[0]
\end{tikzcd}
\end{equation*}
in $\mathsf{Top}_*^\bN$, where $\alpha$ is given by the $\gS$-module structure on $\gR$, and $\beta$ is induced by the augmentation $\epsilon : \gS \to S^{0}[0]$. The image of $\gR \otimes \mathrm{ker}(\epsilon)(n) \to \gR(n) = [(L^+)^{\wedge n}]_{\fS_n}$ is precisely the image of $(L^+)^{\wedge n-2} \wedge (L \oplus L)^+ \to [(L^+)^{\wedge n}]_{\fS_n}$, whose cofibre is by definition $\gC(M;L)$.
\end{proof}

\begin{proof}[Proof of Theorem \ref{thm:MainC}]
Apply Lemma \ref{lem:cof} to the weak equivalence $B(\gS, \gS, S^{0}[0]) \overset{\sim}\to S^{0}[0]$, giving an equivalence $B(\gR, \gS, S^{0}[0]) \overset{\sim}\to \gR \otimes_\gS S^{0}[0]$, and the latter is isomorphic to $\gC(M)$ by Lemma \ref{lem:tensor}.
\end{proof}

\begin{rem}
It is possible to fool oneself into thinking that the above argument can be adapted to the case of ordered configuration spaces, considered in the category of symmetric sequences of pointed spaces, in order to prove a statement analogous to the equivalence \eqref{eq:BarConst} in this category. Unfortunately, that statement is false. One can verify this directly in the case $M=*$ with trivial 0-dimensional Euclidean bundle, in grading 3. If there is an analogue for ordered configuration spaces, its statement must be more complicated.
\end{rem}

\appendix

\section{Homological densities\\ by Quoc P.~Ho and Oscar Randal-Williams}

\subsection{Spaces of 0-cycles}
It is easy to generalise Theorem \ref{thm:MainC} to the following variant of configuration spaces, called ``spaces of 0-cycles'' by Farb--Wolfson--Wood \cite{FWW}. Let $m, k \geq 1$, and for $n_1, n_2, \ldots, n_m \in \bN$ let
$$Z^k_{n_1, \ldots, n_m}(M)  \subset \mathrm{Sym}_{n_1, \ldots, n_m}(M) := [M^{n_1}]_{\fS_{n_1}} \times [M^{n_2}]_{\fS_{n_2}} \times \cdots \times [M^{n_m}]_{\fS_{n_m}}$$
be the open subspace of those $(\{x_1^1, \ldots, x_{n_1}^1\}, \{x_1^2, \ldots, x_{n_2}^2\}, \ldots, \{x_1^m, \ldots, x_{n_m}^m\})$ such that no $x^i_j$ has multiplicity $\geq k$ in all of these $m$ multisets. That is, $Z^k_{n_1, \ldots, n_m}(M)$ is the configuration space of particles of $m$ different colours, $n_i$ having colour $i$, which may collide except that no point of $M$ may carry $\geq k$ points of every colour. The 1-point compactifications $Z_{n_1, \ldots, n_m}(M)^+$ again have a composition product
$$Z^k_{n_1, \ldots, n_m}(M)^+ \wedge Z^k_{n'_1, \ldots, n'_m}(M)^+ \lra Z^k_{n_1+n'_1, \ldots, n_m+n'_m}(M)^+,$$
giving a commutative monoid $\gZ^{m,k}(M)$ in $\bN^m$-graded pointed spaces. Just as before, we can introduce labels in a vector bundle $L \to M$, giving $Z^k_{n_1, \ldots, n_m}(M;L)$ and $\gZ^{m,k}(M;L)$. Writing $1_i = (0,\ldots, 0,1,0\ldots,0) \in \bN^m$ with the 1 in the $i$th position, there is a pushout square
\begin{equation}\label{eq:PushoutZmk}
\begin{tikzcd}
\Com([(L^{\oplus mk})^+]_{\fS_k^m}[k,k,\ldots,k]) \rar{\epsilon} \dar{\Delta} & S^{0}[0, \ldots, 0] \dar \\
{\displaystyle \Com(\bigvee_{i=1}^m L^+[1_i])} \rar &\gZ^{m,k}(M;L)
\end{tikzcd}
\end{equation}
of unital commutative monoids in $\mathsf{Top}_*^{\bN^m}$, where $\Delta$ is now induced by the inclusion $[(L^{\oplus mk})^+]_{\fS_k^m} \to [(L^+)^{\wedge k}]
_{\fS_k} \wedge \cdots \wedge [(L^+)^{\wedge k}]
_{\fS_k} = \Com(\bigvee_{i=1}^m L^+[1_i])(k,\ldots,k)$. The same argument as Theorem \ref{thm:MainC} shows that there is an equivalence
\begin{equation}\label{eq:HomDensity}
\Com(\vee_{i=1}^m L^+[1_i]) \otimes^\bL_{\Com([(L^{\oplus mk})^+]_{\fS_k^m}[k,\ldots,k])} S^{0}[0, \ldots,0] \overset{\sim}\lra  \gZ^{m,k}(M;L).
\end{equation}

\subsection{Revisiting homological densities}
This can be used to revisit the work of Farb--Wolfson--Wood \cite{FWW} and Ho \cite{Ho} on homological densities, and in particular to explain coincidences of homological densities at the level of topology rather than algebra, as proposed in \cite[1.5.1]{Ho}.

The spaces $Z^k_{n_1, \ldots, n_m}(M;L)$ are $\bQ$-homology manifolds, being open subspaces of a product of coarse moduli spaces $[L^n]_{\fS_n}$ of the orbifolds $L^n /\!\!/ \fS_n$. As before, we suppose $M$ is $d$-dimensional and take $L=W$ to be given by the sum of the orientation line of $M$ plus $(d-1)$ trivial lines: then the $Z^k_{n_1, \ldots, n_m}(M;W)$ are orientable $\bQ$-homology manifolds, of dimension $2d \cdot\sum n_i$. Again they are vector bundles over $Z^k_{n_1, \ldots, n_m}(M)$, so Poincar{\'e} duality gives
$$H^*(Z^k_{n_1, \ldots, n_m}(M)) \cong H^*(Z^k_{n_1, \ldots, n_m}(M;W)) \cong\widetilde{H}_{2d\sum n_i - *}(Z^k_{n_1, \ldots, n_m}(M;W)^+).$$
On the other hand, the bar construction formula above together with the argument of Section \ref{sec:Knudsen} identifies the multigraded vector space $H_{*,*}(\gZ^{m,k}(M;W))$ with
$$\mathrm{Tor}_*^{S^*(\widetilde{H}_*([(W^{\oplus mk})^+]_{\fS_k^m})[k,\ldots,k])}(S^*(\bigoplus_{i=1}^m\widetilde{H}_*(W^+)[1_i]), \bQ[0,\ldots,0]).$$

\subsubsection{Odd-dimensional manifolds}\label{sec:odddim}
As in Section \ref{sec:Knudsen} we have $\widetilde{H}_*([(W^{\oplus mk})^+]_{\fS_k^m}) \cong [\Sigma^{dmk} \widetilde{H}_*(M^+)]_{\fS_k^m}$ by the Thom isomorphism. If $d$ is odd then the permutation group $\fS_k^m$ acts on the Thom class via $\fS_k^m\leq \fS_{mk} \overset{sign}\to \bZ^\times$, so acts nontrivially if $k \geq 2$ and trivially if $k=1$. If $k \geq 2$ this means that $\widetilde{H}_*([(W^{\oplus mk})^+]_{\fS_k^m})=0$, showing that
$$H_{*,*}(\Com(\vee_{i=1}^m W^+[1_i])) \overset{\sim}\lra H_{*,*}(\gZ^{m,k}(M;W))$$
in this case. Using Poincar{\'e} duality on both sides gives \cite[Theorem 1.4]{FWW}, except that that theorem is erroneously claimed for all $k \geq 1$. We will return to the case $k=1$ below.

\subsubsection{Even-dimensional manifolds}
If $d$ is even then $\fS_k^m$ acts trivially on $\Sigma^{dmk} \widetilde{H}_*(M^+)$, and using the Thom isomorphism to identify $\widetilde{H}_*(W^+) \cong \Sigma^d \widetilde{H}_*(M^+)$ too, the Koszul complex for computing the $\mathrm{Tor}$-groups above is
$$(S^*\big(\bigoplus_{i=1}^m\Sigma^d \widetilde{H}_*(M^+; \bQ^{w_1})[1_i] \oplus \Sigma^{dmk+1}\widetilde{H}_*(M^+ ; (\bQ^{w_1})^{\otimes mk})[k,\ldots,k]\big), \partial).$$
The differential $\partial$ is induced by the map
$$\Sigma^{dmk}\widetilde{H}_*(M^+ ; (\bQ^{w_1})^{\otimes mk}) \to S^k(\Sigma^d \widetilde{H}_*(M^+; \bQ^{w_1})) \otimes \cdots \otimes S^k(\Sigma^d \widetilde{H}_*(M^+; \bQ^{w_1}))$$
obtained by linearly dualising the cup product map
\begin{equation}\label{eq:CupProd}
H^*_c(M ; \bQ^{w_1})^{\otimes mk} \lra  H^*_c(M ; (\bQ^{w_1})^{\otimes mk}),
\end{equation}
and so is trivial if (and only if) all $mk$-fold cup products of ($w_1$-twisted) compactly-supported cohomology classes on $M$ vanish. 

When this cup product map is trivial, so $\partial$ is trivial, the above just gives a formula for $H_{*,*}(\gZ^{m,k}(M;W))$. Using Poincar{\'e} duality, and reindexing, to express this in terms of $H^*(Z^k_{n_1, \ldots, n_m}(M))$ and $H^*(\mathrm{Sym}_{n_1, \ldots, n_m}(M))$ we obtain an identity of multigraded vector spaces
$$H^*(Z^k_{\bullet}(M)) \cong H^*(\mathrm{Sym}_{\bullet}(M)) \otimes S^*( \Sigma^{d(mk-1)-1}{H}^*(M ; (\bQ^{w_1})^{\otimes mk-1})[k,\ldots,k]).$$
There are stabilisation maps $\sigma_i : H^*(Z^k_{n_1, \ldots, n_m}(M)) \to H^*(Z^k_{n_1, \ldots,n_i+1,\ldots, n_m}(M))$ analogous to those constructed in Section \ref{sec:HomStab}, similarly for $H^*(\mathrm{Sym}_{n_1, \ldots, n_m}(M))$, and both stabilise as $n_j \to \infty$, just as in Section \ref{sec:HomStab}: this recovers \cite[Theorem 1.7]{FWW}. We may take the colimit of all these stabilisations to obtain 
$$H^*(Z^k_{\infty, \ldots, \infty}(M)) \cong H^*(\mathrm{Sym}_{\infty, \ldots, \infty}(M)) \otimes S^*( \Sigma^{d(mk-1)-1}{H}^*(M ; (\bQ^{w_1})^{\otimes mk-1})).$$
Writing $P_{Z^{m,k}}(t)$ and $P_{Sym^m}(t)$ for the Poincar{\'e} series of $H^*(Z^k_{\infty, \ldots, \infty}(M)) $ and $ H^*(\mathrm{Sym}_{\infty, \ldots, \infty}(M))$ respectively, this discussion identifies the \emph{homological density} ${P_{Z^{m,k}}(t)}/{P_{Sym^m}(t)}$ with the Poincar{\'e} series of $S^*( \Sigma^{d(mk-1)-1}{H}^*(M ; (\bQ^{w_1})^{\otimes mk-1}))$. This visibly only depends on the product $mk$, giving ``coincidences between homological densities'': this recovers \cite[Theorem 1.2]{FWW}; in fact it also recovers the stronger Theorem 3.6 of that paper.

\subsubsection{Odd-dimensional manifolds, $k=1$} Just as in the even-dimensional case, if the cup product map \eqref{eq:CupProd} is zero then one gets an explicit description of $H^*(Z_{\bullet}(M))$, and the homological density is given by the Poincar{\'e} series of the graded vector space $S^*( \Sigma^{d(m-1)-1}{H}^*(M ; (\bQ^{w_1})^{\otimes m-1}))$. It follows from Section \ref{sec:odddim} that the homological density is $1$ for $k > 1$, so for odd-dimensional manifolds it is \emph{not true} that the homological density depends only on $mk$.

\subsubsection{Euler characteristic}
If the cup product map \eqref{eq:CupProd} is not zero, and either $d$ is even or $d$ is odd and $k=1$, then there is instead a nontrivial differential on the multigraded vector space
$$H^*(\mathrm{Sym}_{\bullet}(M)) \otimes S^*( \Sigma^{d(mk-1)-1}{H}^*(M ; (\bQ^{w_1})^{\otimes mk-1})[k,\ldots,k]),$$
of degree $(+1,0)$, whose homology is $H^*(Z_{*, \ldots, *}(M))$. Then one would not expect $\frac{P_Z(t)}{P_{Sym}(t)}$ to agree with the Poincar{\'e} series of $S^*( \Sigma^{d(mk-1)-1}{H}^*(M ; (\bQ^{w_1})^{\otimes mk-1}))$, and indeed it does not \cite[Remark 1.6]{FWW}. However, as Euler characteristic commutes with taking homology we have the identity
$$\sum_{n_1, \ldots, n_m \geq 0}  \chi(Z_{n_1, \ldots, n_m}(M)) s_1^{n_1} \cdots s_m^{n_m} = \big(\prod_{i=1}^m (1-s_i)\big)^{-\chi(M)} \cdot (1-(s_1 \cdots s_m)^k)^{\chi(M, (\bQ^{w_1})^{\otimes mk-1})}$$
in $\bZ[[s_1, \ldots, s_m]]$. The left-hand factor is $\sum_{n_1, \ldots, n_m \geq 0}  \chi(\mathrm{Sym}_{n_1, \ldots, n_m}(M)) s_1^{n_1} \cdots s_m^{n_m}$. This recovers \cite[Theorem 1.9 1.]{FWW}. 

\subsection{Spectral densities}

The construction of homological densities can be promoted to the level of spectra, addressing \cite[1.5.1]{Ho}, as follows. Let us \emph{assume that $M$ is even-dimensional and orientable}: then we can dispense with twisting by the vector bundle $W \to M$. We consider $\gZ^{m,k}(M)$ with its $\bN^m$-grading reduced to an $\bN$-grading via $\text{sum} : \bN^m \to \bN$. Collapsing the complement of a small neighbourhood of a point in $M$ gives a map $M^+ \to S^{d}$, inducing a map of commutative monoids
$$\Com(\bigvee_{i=1}^m M^+ [1]) \lra \Com(S^d[1]).$$
If $X$ is a left $\Com(S^d[1])$-module, it is equipped with maps $S^d \wedge X(n) \to X(n+1)$ and so we can define the spectrum $\overline{X} := \hocolim_{n \to \infty} S^{-nd} \wedge \Sigma^\infty X(n)$. Using these two constructions we may therefore form the spectrum
$$\Delta^{m,k} := \overline{\Com(S^d[1]) \otimes^\bL_{\Com(\bigvee_{i=1}^m M^+ [1])} \gZ^{m,k}(M)}.$$

By analogy with \cite[Section 7.5]{Ho} we propose $\Delta^{m,k}$ as a spectral form of the stable density of $Z^k_{n_1, \ldots, n_m}(M)$ in $\mathrm{Sym}_{n_1, \ldots, n_m}(M)$. At the level of $\bQ$-chains it recovers the construction from the proof of Theorem 7.5.1 of \cite{Ho}. We can prove the spectral form of that theorem analogously: as $\bN$-graded objects, there is an evident map from \eqref{eq:PushoutZmk} to the analogous square for $\gZ^{1,mk}(M)$ which induces a map of spectra $\Delta^{m,k} \to \Delta^{1,mk}$, and this is an equivalence by \eqref{eq:HomDensity} as both are identified with $\overline{\Com(S^d[1]) \otimes^\bL_{\Com(M^+ [mk])} S^0[0]}$. 

This may be simplified for $mk \geq 2$ as follows. The map $M^+ \to [(S^d)^{\wedge mk}]_{\fS_{mk}}$ with which the derived tensor product is formed factors over $(S^d)^{\wedge mk}$ so is nullhomotopic when $mk \geq 2$, and so $\Com(S^d[1]) \otimes^\bL_{\Com(M^+ [mk])} S^0[0] $ is equivalent to $ \Com(S^d[1]) \otimes (S^0[0] \otimes^\bL_{\Com(M^+ [mk])} S^0[0])$ as a left $\Com(S^d[1])$-module. In this situation the $\overline{(-)}$ construction gives 
\begin{align*}
\Delta^{m,k} &\simeq \bigvee_{n \geq 0} S^{-nd} \wedge \Sigma^\infty(S^0[0] \otimes^\bL_{\Com(M^+ [mk])} S^0[0])(n)\\
&\simeq \bigvee_{n \geq 0} S^{-nd} \wedge \Sigma^\infty\Com(S^1 \wedge M^+[mk])(n).
\end{align*}

\bibliographystyle{amsalpha}
\bibliography{MainBib}

\providecommand{\bysame}{\leavevmode\hbox to3em{\hrulefill}\thinspace}
\providecommand{\MR}{\relax\ifhmode\unskip\space\fi MR }
\providecommand{\MRhref}[2]{%
  \href{http://www.ams.org/mathscinet-getitem?mr=#1}{#2}
}
\providecommand{\href}[2]{#2}
\begin{thebibliography}{FWW19}

\bibitem[Arn70]{Arnold}
V.~I. Arnol'd, \emph{Certain topological invariants of algebraic functions},
  Trudy Moskov. Mat. Ob\v{s}\v{c}. \textbf{21} (1970), 27--46.

\bibitem[Ban]{Banerjee2}
O.~Banerjee, \emph{Stability in cohomology via the symmetric simplicial
  category}, in preparation.

\bibitem[Ban23]{Banerjee}
\bysame, \emph{Filtration of cohomology via symmetric semisimplicial spaces},
  \url{https://arxiv.org/abs/1909.00458v3}, 2023.

\bibitem[BCM93]{BCM}
C.-F. B\"{o}digheimer, F.~R. Cohen, and R.~J. Milgram, \emph{Truncated
  symmetric products and configuration spaces}, Math. Z. \textbf{214} (1993),
  no.~2, 179--216.

\bibitem[Bia20]{BianchiSplit}
A.~Bianchi, \emph{Splitting of the homology of the punctured mapping class
  group}, J. Topol. \textbf{13} (2020), no.~3, 1230--1260.

\bibitem[BM14]{BenderskyMiller}
M.~Bendersky and J.~Miller, \emph{Localization and homological stability of
  configuration spaces}, Q. J. Math. \textbf{65} (2014), no.~3, 807--815.

\bibitem[BY21]{BerceanuYameen}
B.~Berceanu and M.~Yameen, \emph{Strong and shifted stability for the
  cohomology of configuration spaces}, Bull. Math. Soc. Sci. Math. Roumanie
  (N.S.) \textbf{64(112)} (2021), no.~2, 159--191.

\bibitem[Chu12]{Church}
T.~Church, \emph{Homological stability for configuration spaces of manifolds},
  Invent. Math. \textbf{188} (2012), no.~2, 465--504.

\bibitem[CP15]{CanteroPalmer}
F.~Cantero and M.~Palmer, \emph{On homological stability for configuration
  spaces on closed background manifolds}, Doc. Math. \textbf{20} (2015),
  753--805.

\bibitem[FWW19]{FWW}
B.~Farb, J.~Wolfson, and M.~M. Wood, \emph{Coincidences between homological
  densities, predicted by arithmetic}, Adv. Math. \textbf{352} (2019),
  670--716.

\bibitem[Get99a]{getzler}
E.~Getzler, \emph{The homology groups of some two-step nilpotent {L}ie algebras
  associated to symplectic vector spaces},
  \url{https://arxiv.org/abs/math/9903147}, 1999.

\bibitem[Get99b]{getzler2}
\bysame, \emph{Resolving mixed {H}odge modules on configuration spaces}, Duke
  Math. J. \textbf{96} (1999), no.~1, 175--203.

\bibitem[Ho20]{Ho2}
Q.~P. Ho, \emph{Higher representation stability for ordered configuration
  spaces and twisted commutative factorization algebras},
  \url{https://arxiv.org/abs/2004.00252}, 2020.

\bibitem[Ho21]{Ho}
\bysame, \emph{Homological stability and densities of generalized configuration
  spaces}, Geom. Topol. \textbf{25} (2021), no.~2, 813--912.

\bibitem[Kal98]{Kallel}
S.~Kallel, \emph{Divisor spaces on punctured {R}iemann surfaces}, Trans. Amer.
  Math. Soc. \textbf{350} (1998), no.~1, 135--164.

\bibitem[KM15]{KupersMiller}
A.~Kupers and J.~Miller, \emph{Improved homological stability for configuration
  spaces after inverting 2}, Homology Homotopy Appl. \textbf{17} (2015), no.~1,
  255--266.

\bibitem[KMT23]{KnudsenMillerTosteson}
B.~Knudsen, J.~Miller, and P.~Tosteson, \emph{Extremal stability for
  configuration spaces}, Mathematische Annalen \textbf{386} (2023), no.~3,
  1695--1716.

\bibitem[Knu17]{Knudsen}
B.~Knudsen, \emph{Betti numbers and stability for configuration spaces via
  factorization homology}, Algebr. Geom. Topol. \textbf{17} (2017), no.~5,
  3137--3187.

\bibitem[Loo23]{Looijenga}
E.~Looijenga, \emph{Torelli group action on the configuration space of a
  surface}, J. Topol. Anal. \textbf{15} (2023), no.~1, 215--222.

\bibitem[Mil69]{Milgram}
R.~J. Milgram, \emph{The homology of symmetric products}, Trans. Amer. Math.
  Soc. \textbf{138} (1969), 251--265.

\bibitem[Pet20]{Petersen}
D.~Petersen, \emph{Cohomology of generalized configuration spaces}, Compos.
  Math. \textbf{156} (2020), no.~2, 251--298.

\bibitem[RW13]{RWConf}
O.~Randal-Williams, \emph{Homological stability for unordered configuration
  spaces}, Q. J. Math. \textbf{64} (2013), no.~1, 303--326.

\bibitem[Seg79]{Segal}
G.~Segal, \emph{The topology of spaces of rational functions}, Acta Math.
  \textbf{143} (1979), no.~1-2, 39--72.

\bibitem[Sta23a]{Stavrou}
A.~Stavrou, \emph{Cohomology of configuration spaces of surfaces as mapping
  class group representations}, Trans. Amer. Math. Soc. \textbf{376} (2023),
  no.~4, 2821--2852.

\bibitem[Sta23b]{StavrouThesis}
\bysame, \emph{Homology of configuration spaces of surfaces as mapping class
  group representations}, Ph.D. thesis, University of Cambridge, 2023.

\bibitem[Yam23]{Yameen}
M.~Yameen, \emph{A remark on extremal stability of configuration spaces}, J.
  Pure Appl. Algebra \textbf{227} (2023), no.~1, Paper No. 107154, 5.

\end{thebibliography}

\end{document}